\documentclass[11pt]{article}

\usepackage{url,geometry}
\usepackage{amsmath,amsthm,amssymb}

\usepackage{mathtools,enumitem,multicol}

\newtheorem{theorem}{Theorem}[section]
\newtheorem{proposition}[theorem]{Proposition}
\newtheorem{lemma}[theorem]{Lemma}
\newtheorem{corollary}[theorem]{Corollary}

\theoremstyle{definition}
\newtheorem{definition}[theorem]{Definition}

\theoremstyle{remark}
\newtheorem{remark}[theorem]{Remark}

\newcommand{\Z}{\ensuremath{\mathbb Z}}
\newcommand{\Q}{\ensuremath{\mathbb Q}}
\newcommand{\R}{\ensuremath{\mathbb R}}
\newcommand{\C}{\ensuremath{\mathbb C}}

\DeclareMathOperator{\Irr}{Irr}
\DeclareMathOperator{\GL}{GL}
\DeclareMathOperator{\SL}{SL}

\DeclareMathOperator{\res}{res}

\newcommand{\seslr}[7]{#1 \longrightarrow #2 \stackrel{#3}{\longrightarrow} #4 \stackrel{#5}{\longrightarrow} #6 \longrightarrow #7}

\newcommand{\sesr}[6]{\seslr{#1}{#2}{}{#3}{#4}{#5}{#6}}
\newcommand{\ses}[5]{\seslr{#1}{#2}{}{#3}{}{#4}{#5}}

\title{Minimal non-solvable Bieberbach groups}

\usepackage{authblk}
\author{R. Lutowski\thanks{Corresponding author: \texttt{rafal.lutowski@ug.edu.pl}} }
\author{A. Szczepański}
\affil{\small
Institute of Mathematics\\Faculty of Mathematics, Physics and Informatics\\University of Gdańsk\\Wita Stwosza 57\\80-308 Gdańsk\\Poland
}

\usepackage[backend=bibtex,arxiv=pdf,isbn=false,url=false,doi=false]{biblatex}
\bibliography{bibl.bib}

\begin{document}
	
\maketitle

\let\oldfootnote=\thefootnote
\renewcommand{\thefootnote}{}
\let\thefootnote=\oldfootnote
\setcounter{footnote}{0}

\begin{abstract}
It has been shown by several authors that there exists a non-solvable Bieberbach group of dimension $15$. In this note we show that this is in fact a minimal dimension for such kind of groups.
\end{abstract}
\section{Introduction}

Let $\Gamma$ be an $n$-dimensional crystallographic group, i.e. a discrete and cocompact subgroup of $\operatorname{Iso}(\R^n) = O(n) \ltimes \R^n$, the group of isometries of an $n$-dimensional euclidean space. By the first Bieberbach theorem (see \cite[Theorem 2.1]{Sz12}), the structure of $\Gamma$ is described by the following short exact sequence
\begin{equation}
\label{eq:ses}
\sesr{0}{L}{\Gamma}{\pi}{G}{1},
\end{equation}
where $G$ is a finite group (a holonomy group of $\Gamma$) and $L$ is a free abelian group of rank $n$ and it is the unique maximal abelian normal subgroup of $\Gamma$. The conjugation of elements of $L$ in $\Gamma$ gives $L$ the structure of a $G$-module and so the map $\varphi \colon G \to \GL(L)$, defined by the formula
\[
\varphi_g(l) = \gamma l \gamma^{-1},
\]
where $g \in G, l \in L$ and $\gamma \in \Gamma$ is such that $\pi(\gamma) = g$, is a homomorphism called \emph{integral holonomy representation of $\Gamma$}.

In the case when $\Gamma$ is in addition torsion-free, it is called a \emph{Bieberbach group}, the orbit space $X = \R^n/\Gamma$ is a Riemannian manifold with zero sectional curvature (a \emph{flat manifold}) and $\pi_1(X) = \Gamma$.

In this short note we deal with a structure of fundamental groups of flat manifolds -- with their solvability to be precise. It is known that there exists a non-solvable Bieberbach group in dimension $15$, see \cite[proof of Theorem 2.1]{HMSS87}. Recently J.A. Hillman asked, whether $15$ is in fact the precise lower bound for the dimension of such groups, see \cite{Hi23}. We show that in fact it is:

\begin{theorem}
\label{thm:main}
The minimal dimension of a non-solvable Bieberbach group is $15$.
\end{theorem}

Since the solvability of a Bieberbach group is fully described by the solvability of its holonomy group, in Section \ref{sec:finite} we will focus on finite non-solvable groups, with additional assumption that they are minimal, i.e. they do not have non-solvable proper subgroups. This allows us to reduce the number of groups to consider. The problem of finding a non-solvable Bieberbach group of minimal dimension can be in theory solved only by computer calculations. In practice, our attempts in larger dimensions failed because of taking too long time or not enough computer resources. We deal with this problem in Section \ref{sec:small} by determining minimal non-solvable groups of order up to one million. Having them at hand, the rest of calculations is quite efficient. Finally, in Section \ref{sec:bieb}, we show that there is no non-solvable Bieberbach group in dimension less than $15$.

\begin{remark}
In the paper we will use an abbreviation \emph{MNS}, which will mean \emph{minimal non-solvable}. It will be used in the context of finite and infinite groups as well.
\end{remark}

\section{Finite minimal non-solvable groups}
\label{sec:finite}

In the poset of non-solvable subgroups of a given \emph{finite} non-solvable group there is a minimal element. Hence the following definition is very natural.

\begin{definition}
A finite non-solvable group $G$ is called \emph{minimal non-solvable} if every proper subgroup of $G$ is solvable. 
\end{definition}

\begin{remark}
In \cite{Th68} Thompson defined a notion of a minimal simple group as a non-abelian finite simple group all of whose proper subgroups are solvable. As we will see, minimal non-solvable groups do not have to be simple, hence we consider strictly larger class of groups.
\end{remark}

\begin{lemma}
\label{lm:perfect}
Any finite MNS group is perfect.
\end{lemma}
\begin{proof}
If the commutator subgroup of $G$ is proper, then $G$ is solvable-by-abelian and hence -- solvable.
\end{proof}

Since all subgroups of a solvable group are solvable, we get:

\begin{lemma}
\label{lm:maximal}
A finite group is MNS if and only if it is non-solvable and its maximal subgroups are solvable.
\end{lemma}

\begin{lemma}
\label{lm:minimal-finite}
Let $G$ be a finite MNS group and let $\varphi \colon G \to H$ be a group epimorphism onto a non-trivial group $H$. Then $H$ is a MNS group.
\end{lemma}
\begin{proof}
Every proper subgroup of $H$ is an epimorphic image of a proper, and hence solvable, subgroup of $G$.
\end{proof}

\begin{lemma}
\label{lm:normal-finite}
Let $G$ be a finite MNS group. Then $G$ has a unique maximal normal subgroup.
\end{lemma}
\begin{proof}
Let $H$ and $K$ be maximal normal subgroups of $G$, such that $K \neq H$. By assumption, $G/H$ is a simple non-abelian group. Since $H \subsetneq HK \lhd G$, we have $G=HK$. By the second isomorphism theorem
\[
G/H = HK/H \cong K/(K \cap H)
\]
is simple and hence $K$ is non-solvable. Since every proper subgroup of $G$ is solvable, we get $K=G$, a contradiction.
\end{proof}

\section{Minimal non-solvable subgroups of $\GL_n(\Z)$}
\label{sec:small}

For the purposes of this article we need to determine conjugacy classes of irreducible MNS subgroups of $\GL_n(\Z)$, for $n \leq 10$. Maximal irreducible subgroups of $\GL_n(\Z)$ for $n\leq 3$ have been calculated already in the nineteenth century, for $n=4$ in 1965 by Dade \cite{D65}, for $5 \leq n \leq 9$ by Plesken in papers \cite{P77I,P77II,P80III,P80IV,P80V} and for $n=10$ by Souvignier in \cite{So94}. A library of those groups is available in GAP \cite{GAP4.12.1}. By Lemma \ref{lm:maximal} this should be enough for calculation of all MNS subgroups of $\GL_n(\Z)$, however higher dimensions may be quite involved in terms of computational resources and hence we take another approach. Nevertheless, our calculations show that

\begin{lemma}
Let $n\leq 10$ be a natural number. There is no MNS subgroup of $\GL_n(\Z)$ of order greater than $10^6$.
\end{lemma}

By Lemma \ref{lm:perfect} every finite MNS group is perfect. Hence we can use the before-mentioned criterion from Lemma \ref{lm:maximal} to the library of finite perfect groups from GAP to find all MNS groups of order up to $10^6$ (\emph{small MNS groups}). We get

\begin{proposition}
There are 159 MNS groups of order less than or equal to $10^6$. They are listed in Tables \ref{table:small_mns_hp} and \ref{table:small_mns}.
\end{proposition}

Knowing small MNS groups, we will gather information about their rational representations. Let $G$ be a group and $\Irr_K(G)$ be the set of characters of irreducible representations of $G$ over the field $K$. By \cite[Corollary 10.2(b)]{Is76}
\[
\Irr_\Q(G) = \{ \chi_\Q : \chi \in \Irr(G) \}.
\]
In the above formula, for every $\chi \in \Irr(G)$, we have
\[
\chi_\Q = m_\Q(\chi) \sum \chi^\sigma,
\]
where $m_\Q(\chi)$ denotes the Schur index of $\chi$ and the sum is taken over $\sigma \in \operatorname{Gal}(\Q(\chi)/\Q)$.

\begin{remark}
The crucial part in the determination $\chi_\Q$ for $\chi \in \Irr(G) = \Irr_\C(G)$ is the Schur index $m_\Q(\chi)$. The calculations are available with the GAP package WEDDERGA \cite{Wedderga4.10.2}. In some cases a more efficient approach based on calculating lower and upper bound for $m_\Q(\chi)$ is taken. The lower bound depends on the Frobenius-Schur index $\nu_2(\chi)$ of $\chi$. It is set to $2$ if $\nu_2(\chi)=-1$ and to $1$ in other cases. The upper bound calculation is based on \cite[Lemma 10.4]{Is76}.
\end{remark}

Using the above remark we were able to compute $\Irr_\Q(G)$ for every small MNS group $G$. Looking at the faithful characters, we get

\begin{proposition}
	\label{prop:dim10}
	If $G$ is a finite MNS irreducible subgroup of $\GL_n(\Z)$ for $n \leq 10$, then $G$ is one of the following groups:
	\begin{enumerate}[label=\alph*)]
		\item $n=4$: $A_5$;
		\item $n=5$: $A_5$;
		\item $n=6$: $A_5, L_3(2)$;
		\item $n=7$: $L_3(2), L_2(8), L_3(2)N2^3$;
		\item $n=8$: $A_52^1 = \SL_2(5), L_3(2), L_3(2)2^1=\SL_2(7), L_2(8)$.
	\end{enumerate}
	In particular:
	\begin{enumerate}[resume,label=\alph*)]
	\item for $n \in \{1,2,3\}$ all finite subgroups of $\GL_n(\Z)$ are solvable;
	\item for $n \in \{9,10\}$ there is no minimal non-solvable irreducible subgroup of $\GL_n(\Z)$.
	\end{enumerate}
\end{proposition}

\section{Minimal non-solvable Bieberbach groups}
\label{sec:bieb}

Assume that an $n$-dimensional Bieberbach group $\Gamma$ is given by the short exact sequence \eqref{eq:ses}.

\begin{definition}
Let $\Gamma$ be a Bieberbach group as above.
We will call $\Gamma$ \emph{minimal non-solvable} (\emph{MNS}), if every subgroup $\Gamma'$ of $\Gamma$ such that
\begin{enumerate}[label=\alph*)]
\item $\Gamma'$ is of smaller dimension than $\Gamma$ or
\item $\Gamma' = \pi^{-1}(H)$ for some proper subgroup $H$ of $G$
\end{enumerate}
is solvable.
\end{definition}

Since we will often use specific cases of \cite[Theorem 2.1]{HMSS87}, \cite[Theorem V.1]{P89} and \cite[Theorem VI.1]{P89}, we put them here in the following form
\begin{corollary}
\label{cor:lower_bound}
Let $G$ be a finite group. Let $m(G)$ be the smallest dimension of a Bieberbach group with holonomy $G$. Then $m(A_5)=m(L_3(2))=15$ and $m(A_5 2^1) \geq 15$. 
\end{corollary}
\begin{proof}
It is enough to use formulas in the before-mentioned theorems, noting that $A_5 \cong L_2(4) \cong L_2(5)$ and $L_3(2) \cong L_2(7)$.
\end{proof}

Theorem \ref{thm:main} is a direct result of the above corollary and the following proposition.
\begin{proposition}
\label{prop:min-non-solvable}
If $\Gamma$ is a minimal non-solvable Bieberbach group, then $n \geq 15$.
\end{proposition}

Directly from the definitions of finite and Bieberbach MNS groups we have

\begin{corollary}
\label{cor:minimal-holonomy}
Holonomy group of a MNS Bieberbach group is MNS.
\end{corollary}

\begin{lemma}
\label{lm:center}
Holonomy representation of a MNS Bieberbach group $\Gamma$ does not have any trivial constituent.
\end{lemma}
\begin{proof}
If the holonomy representation of $\Gamma$ has a trivial constituent, then by \cite[Proposition 1.4]{HS86} we have a short exact sequence
\[
\ses{1}{\Gamma'}{\Gamma}{\Z}{0}.
\]
By construction $\Gamma'$ is of dimension less than the dimension of $\Gamma$ and by assumption -- solvable. Hence $\Gamma$ is solvable, a contradiction.
\end{proof}

\begin{lemma}
\label{lm:holonomy_rep}
Let $\Gamma$, given by \eqref{eq:ses}, be a MNS Bieberbach group. Let 
\begin{equation}
\label{eq:decomposition}
\Q \otimes_\Z L = L_1 \oplus \ldots, \oplus L_k
\end{equation}
be a decomposition of $\Q G$-module into irreducible components. Let $\rho_i\colon G \to \GL(L_i)$ be a representation associated with the module $L_i$. Then:
\begin{enumerate}[label=\alph*)]
	\item $\dim_\Q(L_i) \geq 4$ for every $1 \leq i \leq k$,
	\item if $\rho_i(G)$ is a simple group for some $i$, then $\bigoplus_{j \neq i}L_j$ is a faithful $\Q G$-module.
\end{enumerate}
\end{lemma}
\begin{proof}
Let $i \in \{1,\ldots,k\}$.  By Lemma \ref{lm:center} $\rho_i$ is non-trivial. Using Proposition \ref{prop:dim10} we get that if $\dim(L_i) < 4$ then $\rho_i(G)$ is solvable and hence -- by Lemma \ref{lm:minimal-finite} -- $G$ is not minimal non-solvable. Using Corollary \ref{cor:minimal-holonomy} we get that $\Gamma$ is not MNS. 

By Lemma \ref{lm:normal-finite}, the kernel of $\rho_i$ is the maximal normal subgroup of $G$ and hence $\ker \rho_j \subset \ker \rho_i$ for $1 \leq j \leq k$. Since the holonomy representation is faithful, we get
\[
\bigcap_{j \neq i} \ker \rho_j \subset \bigcap \ker \rho_j = 1.
\]
\end{proof}

\begin{corollary}
\label{cor:dim12}
Let $\Gamma$ be a minimal non-solvable Bieberbach group of dimension $n$. Then $n > 12$.
\end{corollary}
\begin{proof}
By \cite[Theorem 1]{Lu21}, the holonomy representation of a non-abelian Bieberbach group contains, over the rationals, at least two non-isomorphic constituents. In particular, it is reducible (see \cite{HS90}). Hence, by Lemma \ref{lm:holonomy_rep}, $n \geq 8$. Let $\Gamma$ be defined by \eqref{eq:ses} with rational holonomy representation decomposition given by \eqref{eq:decomposition}. Using Proposition \ref{prop:dim10} and Lemma \ref{lm:normal-finite} we get that if $8 \leq n \leq 12$ and $k=2$, then the possibilities for $\{\dim L_1, \dim L_2\}$ and $G$ are as follows:
\begin{enumerate}[label=\alph*)]
\item $\{4,4\},\{4,5\},\{4,6\},\{5,5\},\{5,6\},\{6,6\}$ and $G=A_5$ or $G=L_3(2)$,
\item $\{4,8\}$ and $G=A_52^1$.
\end{enumerate}
By Corollary \ref{cor:lower_bound} we get that $n \geq 15$, a contradiction.

If $k=3$ then $\dim L_i = 4$ and $\ker\rho_i$ is the maximal subgroup of $G$, for $i=1,2,3$. Hence all the kernels are equal and -- because the holonomy representation $\rho_1 \oplus \rho_2 \oplus \rho_3$ is faithful -- they are trivial. In that case $G=A_5$ and -- as above -- $n \geq 15$.   
\end{proof}

From \cite[Lemmas 2.1 and 2.2(a)]{HS90} one gets the following lemma. For the sake of completeness, we give the proof here, following the before-mentioned lemmas.
\begin{lemma}
\label{lm:p-adic}
Let $L$ be a $G$-lattice, $U$ be a subgroup of $G$ of prime order $p$ and $\alpha \in H^2(G,L)$. If $\res^G_U(\alpha) \neq 0$ then at least one constituent of $\C \otimes_\Z L$ lies in the principal $p$-block of $G$.
\end{lemma}
\begin{proof}
Let $\Z_p$ denote the ring of $p$-adic integers and $i \colon L \to \Z_p \otimes_\Z L$ be the inclusion. By \cite[Remark II.1(ii)]{P89}, the restriction $\res^G_U i_*(\alpha) \neq 0$ and hence $H^2(G,\Z_p \otimes_\Z L) \neq 0$. By \cite[Lemma 2.2.25]{HP89} some direct summand $U$ of $\Z_p \otimes_\Z L$ lies in the principal $\Z_pG$-block.

By the Brauer theorem $\Q(\zeta)$ is a splitting field for $G$, where $\zeta = \exp{2\pi i/|G|}$ (see \cite[(10.3)]{Is76}). Let $\varphi \colon \Q(\zeta) \to \overline{\Q_p}$ be an embedding over $\Q$, where $\overline{\Q_p}$ is the algebraic closure of the field $\Q_p$ of $p$-adic numbers. Let $K=\varphi(\Q(\zeta))\Q_p$ and $R$ be its ring of integers. Then $\varphi$ induces a bijection between $\Irr(G)$ and $\Irr_K(G)$, which preserves $p$-blocks (see \cite[(7.10)]{Go80}). We finish by noting that there is a direct summand of $R \otimes_{\Z_p} U$ which lies in the principal $p$-block (see \cite[Section VI.1]{F82}).
\end{proof}

We are ready to prove Proposition \ref{prop:min-non-solvable} and hence -- the main theorem.

\begin{proof}[Proof of Proposition \ref{prop:min-non-solvable}]
Let $\Gamma$ be as in \eqref{eq:ses}, with the rational holonomy module decomposition into $k$ irreducible summands as in \eqref{eq:decomposition}. Denote by $N$ the maximal normal subgroup of the holonomy group $G$. Note that we are left with the cases $n=13$ or $n=14$.

Let $k=2$ and $\dim L_1 \leq \dim L_2$. If $\dim L_1 \leq 5$ then by Proposition \ref{prop:dim10} we get $\dim L_1=5, \dim L_2=8$ and $G=A_52^1$, but then, by Corollary \ref{cor:lower_bound}, the dimension of $\Gamma$ is at least $15$. If $\dim L_1 = 6$, then $G/N \cong L_3(2)$ and either $N=1$, which is excluded -- again by Corollary \ref{cor:lower_bound} -- or $N=C_2^3$ and $G=L_3(2)N2^3$, see Proposition \ref{prop:dim10}. In order to get a faithful $G$-lattice of rank $13$ or $14$ one has to use one of two irreducible characters of degree $7$, but in that case none of the possible characters of $G$ lies in the principal $3$-block, see Table \ref{table:1344:2}. If $\dim L_1=7$ then two of the three possibilities for $G$ are the same as before -- namely $L_3(2),L_3(2)N2^3$ -- and excluded for the same reasons. In the third case $G=L_2(8)$ and we get only one rational representation of dimension $7$, hence it is excluded by \cite[Theorem 1]{Lu21}.

If $k=3$, then the possibilities for $(\dim L_1, \dim L_2, \dim L_3)$ -- assuming non-decreasing order -- are as follows: $(4,4,5), (4,4,6), (4,5,5)$. By Lemma \ref{lm:normal-finite} and Proposition \ref{prop:dim10} $G=A_5$, for which $n \geq 15$ by Corollary \ref{cor:lower_bound}.
\end{proof}

\section{Two non-solvable Bieberbach groups}

In this section we give explicit constructions of Bieberbach groups with holonomy groups $A_5$ and $L_3(2)$. Let $\Gamma$ be one of those groups, with holonomy $G$. The integral holonomy representation of $\Gamma$ is a direct sum of irreducible \emph{left} $G$-lattices $L_1,\ldots,L_k$, where $L_i = \Z^{n_i}$ for $i=1,\ldots,k$. We give matrices of actions of generators as well as the images of representatives $\delta_1,\ldots,\delta_k$ of cohomology classes under the isomorphisms $H^2(G,\Z^{n_i}) \cong H^1(G,\Q^{n_i}/\Z^{n_i})$. In addition, for each pair $(L_i,\delta_i)$ we give prime numbers $p$ such that restriction of cohomology class of $\delta_i$ to a subgroup of $G$ of order $p$ is non-zero. All the data show that in fact we define a torsion-free crystallographic group.

\begin{remark}
All representations, except one, have been taken from \cite{AtlasRep2.1.6}. The construction of the exception is explicitly presented.
\end{remark}

\subsection{Bieberbach group with holonomy $A_5$}

The presentation of $A_5$ is as follows:
\[
A_5 = \langle a,b \;|\; a^2=b^3=(ab)^5=1 \rangle
\]
\begin{enumerate}
	\item Lattice $L_1$, $p=3$.
	\begin{align*}
		a \mapsto
		\begin{pmatrix*}[r]
			  1 &   0 &   0 &   0\\
			0 &   0 &   1 &   0\\
			0 &   1 &   0 &   0\\
			-1 &  -1 &  -1 &  -1\\
		\end{pmatrix*}
		\quad
		&
		\quad
		b \mapsto
		\begin{pmatrix*}[r]
			 0 &  1 &  0 &  0\\
			0 &  0 &  0 &  1\\
			0 &  0 &  1 &  0\\
			1 &  0 &  0 &  0\\
		\end{pmatrix*}
	\end{align*}
Cohomology class:
\begin{align*}
	a \mapsto
	\begin{pmatrix}
		0 & \frac{1}{3} & \frac{2}{3} & \frac{2}{3}\\
	\end{pmatrix}^T
	\quad
	&
	\quad
	b \mapsto
	\begin{pmatrix}
		\frac{2}{3} & \frac{1}{3} & \frac{1}{3} & 0
	\end{pmatrix}^T
\end{align*}

	\item Lattice $L_2$, $p=2$.
\begin{align*}
	a \mapsto
	\begin{pmatrix*}[r]
  1 &   0 &   0 &   0 &   0\\
0 &   0 &   1 &   0 &   0\\
0 &   1 &   0 &   0 &   0\\
0 &   0 &   0 &   1 &   0\\
-1 &  -1 &  -1 &  -1 &  -1\\
	\end{pmatrix*}
	\quad
	&
	\quad
	b \mapsto
	\begin{pmatrix*}[r]
  0 &   1 &   0 &   0 &   0\\
0 &   0 &   0 &   1 &   0\\
0 &   0 &   0 &   0 &   1\\
1 &   0 &   0 &   0 &   0\\
-1 &  -1 &  -1 &  -1 &  -1\\
	\end{pmatrix*}
\end{align*}
Cohomology class:
\begin{align*}
	a \mapsto
	\begin{pmatrix}
		\frac{1}{2} & 0 & 0 & \frac{1}{2} & 0\\
	\end{pmatrix}^T
	\quad
	&
	\quad
	b \mapsto
	\begin{pmatrix}
	0 & \frac{1}{2} & 0 & \frac{1}{2} & 0\\
	\end{pmatrix}^T
\end{align*}

	\item Lattice $L_3$, $p=5$.

\begin{align*}
	a \mapsto
	\begin{pmatrix*}[r]
 -1 &  0 &  0 &   0 &  0 &  0\\
0 &  0 &  1 &   0 &  0 &  0\\
0 &  1 &  0 &   0 &  0 &  0\\
0 &  0 &  0 &  -1 &  0 &  0\\
0 &  0 &  0 &   0 &  0 &  1\\
0 &  0 &  0 &   0 &  1 &  0\\
	\end{pmatrix*}
	\quad
	&
	\quad
	b \mapsto
	\begin{pmatrix*}[r]
 0 &  1 &  0 &  0 &  0 &  0\\
0 &  0 &  0 &  1 &  0 &  0\\
0 &  0 &  0 &  0 &  1 &  0\\
1 &  0 &  0 &  0 &  0 &  0\\
0 &  0 &  0 &  0 &  0 &  1\\
0 &  0 &  1 &  0 &  0 &  0\\
	\end{pmatrix*}
\end{align*}
Cohomology class:
\begin{align*}
	a \mapsto
	\begin{pmatrix}
		\frac{1}{5} &  \frac{2}{5} &  \frac{3}{5} &  \frac{2}{5} &  \frac{3}{5} &  \frac{2}{5}\\
	\end{pmatrix}^T
	\quad
	&
	\quad
	b \mapsto
	\begin{pmatrix}
		\frac{3}{5} &  \frac{3}{5} &  \frac{2}{5} &  \frac{4}{5} &  \frac{1}{5} &  \frac{2}{5}\\
	\end{pmatrix}^T
\end{align*}
\end{enumerate}

\subsection{Bieberbach group with holonomy $L_3(2)$}

The presentation of $L_3(2)$ is as follows:
\[
L_3(2) = \langle a,b \;|\; a^2=b^3=(ab)^7=[a,b]^4=1 \rangle
\]

\begin{enumerate}
	\item Lattice $L_1$, $p=2,3$. This is a sublattice of the $7$-dimensional one from \cite{AtlasRep2.1.6} with basis
	\[
	\begin{pmatrix*}[r]
1 &  0 &  0 &  0 &  0 &  0 &  1\\
0 &  1 &  0 &  0 &  0 &  0 &  1\\
0 &  0 &  1 &  0 &  0 &  0 &  1\\
0 &  0 &  0 &  1 &  0 &  0 &  1\\
0 &  0 &  0 &  0 &  1 &  0 &  1\\
0 &  0 &  0 &  0 &  0 &  1 &  1\\
0 &  0 &  0 &  0 &  0 &  2 &  0\\
	\end{pmatrix*}.
	\]
	\begin{align*}
		a \mapsto
		\begin{pmatrix*}[r]
-1 &   0 &  -1 &  -1 &  -1 &   3 &  -2\\
 0 &  -1 &  -1 &  -1 &  -1 &   3 &  -2\\
-1 &  -1 &  -1 &   0 &  -1 &   3 &  -2\\
-1 &  -1 &   0 &  -1 &  -1 &   3 &  -2\\
-1 &  -1 &  -1 &  -1 &  -1 &   4 &  -2\\
-1 &  -1 &  -1 &  -1 &   0 &   3 &  -2\\
 0 &   0 &   0 &   0 &   2 &  -2 &   1\\
		\end{pmatrix*}
		\quad
		&
		\quad
		b \mapsto
		\begin{pmatrix*}[r]
0 &  0 &  1 &  1 &  0 &  -2 &  1\\
1 &  0 &  0 &  1 &  0 &  -2 &  1\\
0 &  1 &  0 &  1 &  0 &  -2 &  1\\
0 &  0 &  0 &  1 &  1 &  -2 &  1\\
0 &  0 &  0 &  1 &  0 &   0 &  0\\
0 &  0 &  0 &  1 &  0 &  -1 &  1\\
0 &  0 &  0 &  0 &  0 &   0 &  1\\
		\end{pmatrix*}
	\end{align*}
	Cohomology class:
	\begin{align*}
		a \mapsto
		\begin{pmatrix}
 \frac{1}{6} &  \frac{2}{3} &  \frac{1}{6} &  \frac{2}{3} &  \frac{2}{3} &  \frac{1}{6} & 0\\
		\end{pmatrix}^T
		\quad
		&
		\quad
		b \mapsto
		\begin{pmatrix}
\frac{2}{3} &  \frac{1}{3} &  \frac{1}{2} &  \frac{2}{3} &  \frac{1}{6} &  \frac{1}{6} & \frac{2}{3}\\
		\end{pmatrix}^T
	\end{align*}

	\item Lattice $L_2$, $p=7$.
\begin{align*}
	a \mapsto
	\begin{pmatrix*}[r]
  0 &  1 &   0 &  0 &  0 &  0 &  0 &  0\\
1 &  0 &   0 &  0 &  0 &  0 &  0 &  0\\
-1 &  0 &  -1 &  1 &  0 &  0 &  0 &  0\\
-1 &  1 &   0 &  1 &  0 &  0 &  0 &  0\\
0 &  0 &   0 &  0 &  0 &  1 &  0 &  0\\
0 &  0 &   0 &  0 &  1 &  0 &  0 &  0\\
0 &  0 &   0 &  0 &  0 &  0 &  0 &  1\\
0 &  0 &   0 &  0 &  0 &  0 &  1 &  0\\
	\end{pmatrix*}
	\quad
	&
	\quad
	b \mapsto
	\begin{pmatrix*}[r]
 0 &   0 &  1 &   0 &  0 &  0 &  0 &  0\\
0 &   0 &  0 &   1 &  0 &  0 &  0 &  0\\
0 &   0 &  0 &   0 &  1 &  0 &  0 &  0\\
1 &  -1 &  1 &  -1 &  1 &  0 &  0 &  0\\
1 &   0 &  0 &   0 &  0 &  0 &  0 &  0\\
0 &   0 &  0 &   0 &  0 &  0 &  1 &  0\\
0 &   1 &  0 &   0 &  0 &  0 &  0 &  1\\
0 &   0 &  0 &  -1 &  0 &  1 &  0 &  0\\
	\end{pmatrix*}
\end{align*}
Cohomology class:
\begin{align*}
	a \mapsto
	\begin{pmatrix}
\frac{4}{7} &  \frac{3}{7} &  \frac{2}{7} &  \frac{4}{7} &  \frac{4}{7} &  \frac{3}{7} &  0 &  0\\
	\end{pmatrix}^T
	\quad
	&
	\quad
	b \mapsto
	\begin{pmatrix}
 \frac{2}{7} &  \frac{6}{7} &  \frac{4}{7} &  0 &  \frac{1}{7} &  \frac{4}{7} &  0 &  \frac{4}{7}\\
	\end{pmatrix}^T
\end{align*}
\end{enumerate}

\section*{Acknowledgments}

The authors would like to thank J.A. Hillman for pointing our attention to the problem of finding non-solvable Bieberbach group of minimal dimension.

\printbibliography

\begin{table}[!h]
	\begin{center}
		\setlength{\columnsep}{-1em}
		\footnotesize
		\begin{multicols}{2}
			\begin{tabular}{rrl}
				Order & Id & Description \\ \hline
				60 & 1 & $A_5$\\
				120 & 1 & $A_5\,2^1$\\                 
				168 & 1 & $L_3(2)$\\     
				336 & 1 & $L_3(2)\,2^1 = SL_2(7)$\\
				504 & 1 & $L_2(8)$\\               
				1092 & 1 & $L_2(13)$\\                  
				1344 & 2 & $L_3(2)\,\mathrm{N}\,2^3$\\            
				1920 & 4 & $A_5\,2^1\,\mathrm{E}\,2^4$\\          
				2184 & 1 & $L_2(13)\,2^1 = SL_2(13)$\\
				2448 & 1 & $L_2(17)$\\                
				2688 & 3 & $L_3(2)\,2^1\times\mathrm{N}\,2^3$\\      
				3840 & 6 & $A_5\,2^1\,\mathrm{E}\,2^4\,\mathrm{E}\,2^1$\\
				4860 & 2 & $A_5\,\mathrm{N}\,3^{4'}$\\
				4896 & 1 & $L_2(17)\,2^1 = SL_2(17)$\\
				5616 & 1 & $L_3(3)$\\
				6072 & 1 & $L_2(23)$\\
				7500 & 2 & $A_5\,\mathrm{N}\,5^3$\\
				9720 & 2 & $A_5\,2^1\times\mathrm{N}\,3^{4'}$\\
				9828 & 1 & $L_2(27)$\\
				10752 & 4 & $L_3(2)\,\mathrm{N}\,2^3\,\mathrm{A}\,2^3$\\
				10752 & 7 & $L_3(2)\,\mathrm{N}\,2^3\times\mathrm{N}\,2^{3'}$\\
				10752 & 9 & $L_3(2)\,\mathrm{N}\,2^3\,\mathrm{E}\,2^{3'}$\\
				12144 & 1 & $L_2(23)\,2^1 = SL_2(23)$\\
				15000 & 2 & $A_5\,2^1\times\mathrm{N}\,5^3$\\
				19656 & 1 & $L_2(27)\,2^1 = SL_2(27)$\\
				21504 & 8 & $L_3(2)\,2^7$\\
				21504 & 16 & $L_3(2)\,2^7$\\
				21504 & 22 & $L_3(2)\,2^1\,\times(\mathrm{N}\,2^3\,\mathrm{E}\,2^{3'}) $\\
				25308 & 1 & $L_2(37)$\\
				29120 & 1 & Sz(8)\\
				30720 & 11 & $A_5\,2^1\,\mathrm{E}\,2^4\,\mathrm{A}\,2^4$\\
				30720 & 22 & $A_5\,2^1\,\mathrm{E}\,2^4$ C $2^{4'}$\\
				32256 & 2 & $L_2(8)\,\mathrm{N}\,2^6$\\
				32736 & 1 & $L_2(32)$\\
				39732 & 1 & $L_2(43)$\\
				43008 & 19 & $L_3(2)\,2^1\,(\mathrm{N}\,2^3 \times \mathrm{N}\,2^{3'})\,\mathrm{E}\,2^1$\\
				50616 & 1 & $L_2(37)\,2^1 = SL_2(37)$\\
				51888 & 1 & $L_2(47)$\\
				57624 & 2 & $L_3(2)\,\mathrm{N}\,7^3$\\
				58240 & 1 & Sz(8)$\,2^1$\\
			\end{tabular}
			
			\begin{tabular}{rrl}
				Order & Id & Description \\ \hline
				64512 & 2 & $L_2(8)\,\mathrm{N}\,2^6\,\mathrm{E}\,2^1$ I\\
				64512 & 3 & $L_2(8)\,\mathrm{N}\,2^6\,\mathrm{E}\,2^1$ II\\
				64512 & 4 & $L_2(8)\,\mathrm{N}\,2^6\,\mathrm{E}\,2^1$ III\\
				74412 & 1 & $L_2(53)$\\
				79464 & 1 & $L_2(43)\,2^1 = SL_2(43)$\\
				103776 & 1 & $L_2(47)\,2^1 = SL_2(47)$\\
				115248 & 2 & $L_3(2)\,2^1\times\mathrm{N}\,7^3$\\
				116480 & 1 & Sz(8) $2^1\times2^1$\\
				129024 & 2 & $L_2(8)\mathrm{N}(2^6\,\mathrm{E}\,2^1$ A ) C $2^1$\\ 
				129024 & 3 & $L_2(8)\,\mathrm{N}\,2^6\mathrm{N}(2^1\times2^1)\,$I\\ 
				129024 & 4 & $L_2(8)\,\mathrm{N}\,2^6\mathrm{N}(2^1\times2^1)\,$II\\
				129024 & 5 & $L_2(8)\,\mathrm{N}\,2^6\mathrm{N}(2^1\times2^1)\,$III\\
				148824 & 1 & $L_2(53)\,2^1 = SL_2(53)$\\
				150348 & 1 & $L_2(67)$\\
				155520 & 12 & $A_5\,\mathrm{\#}\,2^5\,3^4\,$[11]\\
				194472 & 1 & $L_2(73)$\\
				240000 & 11 & $A_5\,\mathrm{\#}\,2^5\,5^3\,$[11]\\
				258048 & 2 & $L_2(8)\,\mathrm{N}\,(\,2^6\,\mathrm{N}\,(2^1\times2^1\,\mathrm{A}\,)\,)\mathrm{C}\,2^1$\\
				258048 & 3 & $L_2(8)\,\mathrm{N}\,2^6\,\mathrm{N}\,(2^1\times2^1\times2^1) $\\
				285852 & 1 & $L_2(83)$\\
				300696 & 1 & $L_2(67)\,2^1 = SL_2(67)$\\
				311040 & 14 & $A_5\,\mathrm{\#}\,2^6\,3^4\,$[13]\\
				367416 & 3 & $L_3(2)\,\mathrm{N}\,3^7$\\
				388944 & 1 & $L_2(73)\,2^1 = SL_2(73)$\\
				393660 & 4 & $A_5\,\mathrm{N}\,3^{4'}\,\mathrm{A}\,3^{4'}$\\
				456288 & 1 & $L_2(97)$\\
				460992 & 4 & $L_3(2)\,\mathrm{\#}\,2^3\,7^3\,$[4]\\
				480000 & 13 & $A_5\,\mathrm{\#}\,2^6\,5^3\,$[13]\\
				516096 & 1 & $L_2(8)\,\mathrm{N}\,(\,2^6\,\mathrm{N}\,(\,2^1\times 2^1\times 2^1\,\mathrm{A}\,)\,)\,\mathrm{C}\,2^1$\\
				546312 & 1 & $L_2(103)$\\
				571704 & 1 & $L_2(83)\,2^1 = SL_2(83)$\\
				607500 & 4 & $A_5\,\mathrm{\#}\,3^4\,5^3\,$[4]\\
				612468 & 1 & $L_2(107)$\\
				721392 & 1 & $L_2(113)$\\
				734832 & 3 & $L_3(2)\,2^1\times\mathrm{N}\,3^7$\\
				787320 & 4 & $A_5\,2^1\times\mathrm{N}\,3^{4'}\,\mathrm{A}\,3^{4'}$\\
				912576 & 1 & $L_2(97)\,2^1 = SL_2(97)$\\
				921984 & 6 & $L_3(2)\,\mathrm{\#}\,2^4\,7^3\,$[6]\\
				937500 & 7 & $A_5\,\mathrm{N}\,5^3\,\mathrm{E}\,5^3$\\
				937500 & 8 & $A_5\,\mathrm{N}\,5^3\,\mathrm{C}\,5^3$\\
			\end{tabular}
		\end{multicols}
	\end{center}
	\caption{MNS groups of order up to $10^6$ listed in \cite{HP89}. Order and id give identification of perfect group in GAP.}
	\label{table:small_mns_hp}
\end{table}

\newcommand\tablefont{\fontsize{7pt}{8pt}\selectfont}
\begin{table}[!h]
	\begin{center}
		\footnotesize
		\begin{multicols}{6}
			\begin{tabular}{lr}
				Order & Id\\ \hline
				61440 & 13\\ 61440 & 14\\ 61440 & 15\\ 61440 & 52\\ 61440 & 53\\ 61440 & 54\\ 
				61440 & 76\\ 86016 & 24\\ 86016 & 25\\ 86016 & 26\\ 86016 & 27\\ 86016 & 28\\
				86016 & 35\\ 86016 & 36\\
			\end{tabular}
			
			\begin{tabular}{lr}
				Order & Id\\ \hline
				 86016 & 40\\ 86016 & 41\\ 122880 & 88\\ 122880 & 89\\
				122880 & 90\\ 122880 & 218\\ 122880 & 219\\ 172032 & 1\\ 172032 & 91\\ 172032 & 92\\
				172032 & 93\\ 172032 & 94\\ 172032 & 95\\
			\end{tabular}
			
			\begin{tabular}{lr}
				Order & Id\\ \hline
				 172032 & 128\\ 172032 & 129\\ 172032 & 151\\
				172032 & 152\\ 245760 & 566\\ 344064 & 191\\ 344064 & 268\\ 344064 & 269\\
				344064 & 291\\ 491520 & 19\\ 491520 & 21\\ 688128 & 176\\ 688128 & 177\\
			\end{tabular}
			
			\begin{tabular}{lr}
				Order & Id\\ \hline
				 688128 & 178\\
				688128 & 179\\ 688128 & 180\\ 688128 & 181\\ 688128 & 182\\ 688128 & 183\\
				688128 & 184\\ 688128 & 185\\ 688128 & 186\\ 688128 & 187\\ 688128 & 188\\ 688128 & 189\\
				688128 & 190\\ 
			\end{tabular}
			
			\begin{tabular}{lr}
				Order & Id\\ \hline	
				
				688128 & 191\\ 688128 & 192\\ 688128 & 226\\ 688128 & 227\\
				688128 & 228\\ 688128 & 229\\ 688128 & 230\\ 688128 & 231\\ 688128 & 232\\ 688128 & 251\\
				688128 & 252\\ 688128 & 253\\ 688128 & 254\\ 
			\end{tabular}
			
			\begin{tabular}{lr}
				Order & Id\\ \hline
				
				983040 & 64\\ 983040 & 66\\
				983040 & 67\\ 983040 & 69\\ 983040 & 72\\ 983040 & 104\\ 983040 & 105\\ 983040 & 361\\
				983040 & 362\\ 983040 & 363\\ 983040 & 371\\ 983040 & 372\\ 983040 & 373\\
			\end{tabular}
		\end{multicols}
	\end{center}
	\caption{MNS groups of order up to $10^6$ not listed in \cite{HP89}. Order and id give identification of perfect group in GAP.}
	\label{table:small_mns}
\end{table}

\begin{table}[ht]
\begin{align*}
& \begin{array}{l|rrrrrrrrrrr|r|rrr}
 & 1a & 2a & 2b & 3a & 4a & 4b & 6a & 7a & 7b & 8a &  8b & m & B_2 & B_3 & B_7 \\ \hline
\chi_{1}  & 1 & 1 & 1 & 1 & 1 & 1 & 1 & 1 & 1 & 1 & 1 & 1 & 1 & 1 & 1\\
\chi_{2}  & 3 & 3 &-1 & 0 &-1 &-1 & 0 & \alpha &\overline{\alpha} & 1 & 1 & 1 & 1 & 2 & 1\\
\chi_{3}  & 3 & 3 &-1 & 0 &-1 &-1 & 0 &\overline{\alpha} & \alpha & 1 & 1 & 1 & 1 & 3 & 1\\
\chi_{4}  & 6 & 6 & 2 & 0 & 2 & 2 & 0 &-1 &-1 & 0 & 0 & 1 & 1 & 4 & 1\\
\chi_{5}  & 7 & 7 &-1 & 1 &-1 &-1 & 1 & 0 & 0 &-1 &-1 & 1 & 1 & 1 & 2\\
\chi_{6}  & 7 &-1 &-1 & 1 & 3 &-1 &-1 & 0 & 0 & 1 &-1 & 1 & 1 & 5 & 3\\
\chi_{7}  & 7 &-1 &-1 & 1 &-1 & 3 &-1 & 0 & 0 &-1 & 1 & 1 & 1 & 5 & 4\\
\chi_{8}  & 8 & 8 & 0 &-1 & 0 & 0 &-1 & 1 & 1 & 0 & 0 & 1 & 1 & 1 & 1\\
\chi_{9}  & 14& -2& -2& -1&  2&  2&  1&  0&  0&  0&  0& 1 & 1 & 5 & 5\\
\chi_{10} & 21& -3&  1&  0&  1& -3&  0&  0&  0& -1&  1& 1 & 1 & 6 & 6\\
\chi_{11} & 21& -3&  1&  0& -3&  1&  0&  0&  0&  1& -1& 1 & 1 & 7 & 7\\
\end{array}
\\
& \alpha = -(1+i\sqrt{7})/2
\end{align*}
\caption{Character table of $L_3(2) N 2^3$. Conjugacy classes are named by the orders of their elements, suffixed by a letter, $m$ denotes the Schur index and $p$-blocks in column $B_p$ are labeled by natural numbers.}
\label{table:1344:2}
\end{table}
\end{document}